\documentclass[10pt]{article}

\usepackage{amssymb}
\usepackage{amsfonts}
\usepackage{amsthm}
\usepackage{amsmath}
\usepackage{float}
\usepackage{bigints}
\usepackage{fullpage}

\usepackage[dvips]{graphicx}   %To insert ps figures
\usepackage{color,epsfig}      %To insert ps figures

\author{ K\'aroly Bezdek 
\thanks{Partially supported by a Natural Sciences and 
Engineering Research Council of Canada Discovery Grant.}}

\font\tenBbb=msbm10 at 10pt         \font\sevenBbb=msbm9    \font\fiveBbb=msbm7
\newfam\Bbbfam
\textfont\Bbbfam=\tenBbb \scriptfont\Bbbfam=\sevenBbb
\scriptscriptfont\Bbbfam=\fiveBbb

\def\S{{\mathbb S}}
\def\E{{\mathbb E}}

\def\oo{{\bf o}}
\def\bb{{\bf b}}
\def\cc{{\bf c}}
\def\aa{{\bf a}}
\def\pp{{\bf p}}
\def\qq{{\bf q}}
\def\tt{{\bf t}}
\def\xx{{\bf x}}
\def\yy{{\bf y}}
\def\KK{{\mathcal K}}
\def\LL{{\mathcal L}}

\def\BB{{\mathcal B}}
\def\PP{{\mathcal P}}

\newtheorem{theorem}{Theorem}[section]

\newtheorem{sublemma}[theorem]{Sublemma}
\newtheorem{remark}[theorem]{Remark}

\newtheorem{proposition}[theorem]{Proposition}

\newtheorem{con}[theorem]{Conjecture}
\newtheorem{prob}[theorem]{Problem}

\newtheorem{definition}[theorem]{Definition}

\title{On a strengthening of the Blaschke-Leichtweiss theorem 
\footnote{ 
\newline \hspace*{.35cm} 2000 Mathematical Subject Classification. Primary: 52A10, 52A38, 52A55, Secondary: 52A20, 52A40.}}

\begin{document}

\maketitle

\date

\begin{abstract}
The Blaschke-Leichtweiss theorem (Abh. Math. Sem. Univ. Hamburg 75: 257--284, 2005) states that the smallest area convex domain of constant width $w$ in the $2$-dimensional spherical space $\S^2$ is the spherical Reuleaux triangle for all $0<w\leq\frac{\pi}{2}$. In this paper we extend this result to the family of wide $r$-disk domains of $\S^2$, where  $0<r\leq\frac{\pi}{2}$. Here a wide $r$-disk domain is an intersection of spherical disks of radius $r$ with centers contained in their intersection. This gives a new and elementary proof of the Blaschke-Leichtweiss theorem. Furthermore, we investigate the higher dimensional analogue of wide $r$-disk domains called wide $r$-ball bodies. In particular, we determine their minimum spherical width (resp., inradius) in the spherical $d$-space $\S^d$ for all $d\geq 2$. Also, it is shown that any minimum volume wide $r$-ball body is of constant width $r$ in $\S^d$, $d\geq 2$.
\end{abstract}

\medskip

\section{Introduction}

\subsection{On the definition of convex bodies of constant width in spherical space}

Let $\S^d=\{\mathbf{x}\in\E^{d+1} \ |\ \|\mathbf{x}\|=\sqrt{\langle\mathbf{x},\mathbf{x}\rangle}=1\}$ be the unit sphere centered at the origin $\oo$ in the $(d+1)$-dimensional Euclidean space $\E^{d+1}$, where $\|\cdot\|$ and $\langle \cdot , \cdot \rangle$ denote the canonical Euclidean norm and the canonical inner product in $\E^{d+1}$, $d\geq2$. A $(d-1)$-dimensional great sphere of $\S^d$ is an intersection of $\S^d$ with a hyperplane of $\E^{d+1}$ passing through $\oo$ (i.e., with a $d$-dimensional linear subspace of $\E^{d+1}$). In particular, an intersection of $\S^d$ with a $2$-dimensional linear subspace in $\E^{d+1}$ is called a great circle of $\S^d$. Two points are called antipodes if they can be obtained as an intersection of $\S^d$ with a line through $\oo$ in $\E^{d+1}$. If $\aa ,\bb\in\S^d$ are two points that are not antipodes, then we label the (uniquely determined) shortest geodesic arc of $\S^d$ connecting $\aa$ and $\bb$ by $\aa\bb$. In other words, $\aa\bb$ is the shorter circular arc with endpoints $\aa$ and $\bb$ of the great circle $\widehat{\aa\bb}$ that passes through $\aa$ and $\bb$. The length of $\aa\bb$ is called the spherical distance between $\aa$ and $\bb$ and it is labelled by ${\rm dist}_s (\aa,\bb)$. Clearly, $0<{\rm dist}_s (\aa,\bb)<\pi$. If $\aa,\bb\in\S^d$ are antipodes, then we set ${\rm dist}_s (\aa,\bb)=\pi$. Let $\xx\in\S^d$ and $r\in (0,\frac{\pi}{2}]$. Then the set 
$${\bf B}_s^d[\xx, r]:=\{\yy\in\S^d\ |\ {\rm dist}_s (\xx,\yy)\leq r\}\ ({\rm resp.,}\ {\bf B}_s^d(\xx, r):=\{\yy\in\S^d\ |\ {\rm dist}_s (\xx,\yy)< r\})$$ 
is called the $d$-dimensional closed (resp., open) spherical ball, or shorter the $d$-dimensional closed (resp., open) ball, centered at $\xx$ having (spherical) radius $r$ in $\S^d$. In particular, ${\bf B}_s^d[\xx, \frac{\pi}{2}]$ (resp., ${\bf B}_s^d(\xx, \frac{\pi}{2})$) is called the closed (resp., open) hemisphere of $\S^d$ with center $\xx$. Moreover, ${\bf B}_s^2[\xx, r]$ (resp., ${\bf B}_s^2(\xx, r)$) is called the closed (resp., open) disk with center $\xx$ and (spherical) radius $r$ in $\S^2$. Now, the boundary of ${\bf B}_s^d[\xx,r]$ (resp., ${\bf B}_s^d(\xx, r)$) in $\S^d$ is called the $(d-1)$-dimensional sphere $S_s^{d-1}(\xx, r)$ with center $\xx$ and (spherical) radius $r$ in $\S^d$. As a special case, the boundary of the disk ${\bf B}_s^2[\xx, r]$ (resp., ${\bf B}_s^2(\xx, r)$) in $\S^2$ is called the circle with center $\xx$ and of (spherical) radius $r$ and it is labelled by $S_s^1(\xx, r)$. We introduce the following additional notations. For a set $X\subseteq\S^d$ and $r\in (0,\frac{\pi}{2}]$ let
$${\bf B}_s^d[X, r]:= \bigcap_{\xx\in X}{\bf B}_s^d[\xx, r]\ {\rm and}\ {\bf B}_s^d(X, r):= \bigcap_{\xx\in X}{\bf B}_s^d(\xx, r).$$
Another basic concept is spherical convexity: we say that $Q\subset\S^d$ is spherically convex if it has no antipodes and for any two points $\xx,\yy\in Q$ we have $\xx\yy\subseteq Q$. (It follows that there exists $\qq\in \S^d$ such that $Q\subseteq {\bf B}_s^d(\qq, \frac{\pi}{2})$ with ${\bf B}_s^d(\xx, \frac{\pi}{2})$ being spherically convex.) As the intersection of spherically convex sets is spherically convex therefore if $X\subset {\bf B}_s^d(\xx, \frac{\pi}{2})$, then we define the spherical convex hull ${\rm conv}_sX$ of $X$ as the intersection of spherically convex sets containing $X$. By a convex body in $\S^d$ (resp., a convex domain in $\S^2$) we mean a closed spherically convex set with non-empty interior in $\S^d$ (resp., in $\S^2$). Let $\KK_s^d, d\geq 2$ denote the family of convex bodies in $\S^d$. If $Q\subseteq\S^d, d\geq 2$, then its spherical diameter is ${\rm diam}_s(Q):=\sup\{{\rm dist}_s(\xx,\yy)\ |\ \xx,\yy\in Q\}$. The following brief definition serves the purpose of our introduction as well as motivates our strengthening of the Blaschke-Leichtweiss theorem in an efficient way. On the other hand, we call the attention of the interested reader to some equivalent definitions  that are discussed in the articles \cite{Be12}, \cite{De95}, \cite{HaWu}, \cite{La15},  \cite{LaMu}, \cite{La20A}, \cite{La20}, \cite{Le05} and are surveyed in \cite{La21}.
\begin{definition}\label{core-concept}
Let ${\bf K}\subset\S^d, d\ge2$ be a closed set with spherical diameter $0<w:={\rm diam}_s({\bf K})\leq\frac{\pi}{2}$. We say that ${\bf K}$ is a convex body of constant width $w$ in $\S^d$ if ${\bf K}={\bf B}_s^d[{\bf K}, w]$. Let $\KK_s^d(w)$ denote the family of convex bodies of constant width $w$ in $\S^d$ for $d\ge 2$ and $0<w\leq\frac{\pi}{2}$.  
\end{definition}

Clearly, $\KK_s^d(w)\subset\KK_s^d$ for all $d\ge 2$ and $0<w\leq\frac{\pi}{2}$. 
%On the one hand, the main results (Theorems 1 and 2) proved in \cite{De95} imply that Definition~\ref{core-concept} is in fact, equivalent to the so-called classical definition of spherical convex bodies of constant width using normal directions, which is discussed in details in Section 1.3 of \cite{De95} (see also \cite{Le05} and in particular, \cite{La15} for another equivalent approach). On the other hand, Definition~\ref{core-concept} helps to take a new look of the Blaschke-Leichtweiss theorem as discussed in the next two subsections. For the sake of completeness, we note that a lot more is known about the geometry of convex bodies of constant width in Euclidean spaces than in spherical (resp., hyperbolic) spaces (see \cite{Ka09}, \cite{KaWe11}, and the recent comprehensive monograph \cite{MaMoOl}). In particular, while the Euclidean analogue of Theorem~\ref{BL-generalized} of this paper has already been proved in \cite{Be11}, finding its hyperbolic analogue remains to be seen.  

\subsection{The Blaschke-Leichtweiss theorem}

The classical Blaschke-Lebesgue theorem states that in the Euclidean plane among all convex sets of constant width $w'>0$ the Reuleaux triangle minimizes area.
Here the Reuleaux triangle is the intersection of three disks of radius $w'$ with centers at the vertices of an equilateral triangle of side length $w'$. For a survey on this
theorem and its impact on extremal geometry we refer the interested reader to the recent elegant papers \cite{Ka09} and \cite{KaWe11}. Very different proofs of 
this theorem were given by Blaschke \cite{Bl15}, Lebesgue \cite{Le14}, Fujiwara \cite{Fu27}, Eggleston \cite{Eg52}, Besicovich \cite{Be63}, Ghandehari \cite{Gh96}, 
Campi, Colesanti, and Gronchi \cite{CCG96}, Harrell \cite{Ha02}, and M. Bezdek \cite{Be11}. So, it is natural to ask whether any of these proofs can be extended 
to $\S^2$. Actually, Blaschke claimed that this can be done with his Euclidean proof (see \cite{Bl15}, p. 505), but one had to wait until Leichtweiss
did it (using some ideas of Blaschke) in \cite{Le05}. So, we call the following statement the Blaschke-Leichtweiss theorem: if ${\bf K}\in\KK_s^2(w)$ with $0<w\leq\frac{\pi}{2}$,
then 
\begin{equation}\label{BL-inequality}
{\rm area}_s({\bf K})\geq{\rm area}_s({\bf \Delta}_2(w)),
\end{equation} 
where ${\rm area}_s(\cdot)$ refers to the spherical area of the corresponding set in $\S^2$ and ${\bf \Delta}_2(w)$ denotes the spherical Reuleaux triangle
which is the intersection of three disks of radius $w$ with centers at the vertices of a spherical equilateral triangle of side length $w$. As the only known proof of (\ref{BL-inequality}) is the one published in \cite{Le05} which is a combination of geometric and analytic ideas presented on twenty pages, one might wonder whether there is a simpler approach. This paper intends to fill this gap by proving a stronger result (Theorem~\ref{BL-generalized}) in a new and elementary way. The Euclidean analogue of Theorem~\ref{BL-generalized} has already been proved for disk-polygons in \cite{Be11} and our proof of Theorem~\ref{BL-generalized} presented below is an extension of the Euclidean technique of \cite{Be11} to $\S^2$ combined with the properly modified spherical method of \cite{BeBl}. For the sake of completeness we note that in \cite{BeBl} the author and Blekherman proved a spherical analogue of P\'al's theorem stating that the minimal spherical area convex domain of given minimal spherical width $\omega$ is a regular spherical triangle for all $0<\omega\leq\frac{\pi}{2}$. 
%(For a definition of minimal spherical width see Definition~\ref{min-spherical-width}.)

\subsection{The Blaschke-Leichtweiss theorem extended: minimizing the area of wide $r$-disk domains in $\S^2$}

The following definition introduces wide $r$-disk domains in $\S^2$ the spherical areas of which we wish to minimize for given $0<r\leq\frac{\pi}{2}$.

\begin{definition}\label{wide-disk-domain-polygon}
Let $0<r\leq\frac{\pi}{2}$ be given and let $\emptyset\neq X$ be a closed subset of $\S^2$ with ${\rm diam}_s(X)\leq r$. Then ${\bf B}_s^2[X, r]$
is called the wide $r$-disk domain generated by $X$ in $\S^2$. The family of wide $r$-disk domains of $\S^2$ is labelled by $\BB^2_{s,{\rm wide}}(r)$. 
%If $X\subset\S^2$ with 
%$0<{\rm card}(X)<+\infty$ and ${\rm diam}_s(X)\leq r$, then ${\bf B}_s^2[X, r]$ is called a wide $r$-disk polygon generated by $X$ in $\S^2$. Finally, the family of 
%spherical wide $r$-disk polygons of $\S^2$ is labelled by $\PP^2_{s, {\rm wide}}(r)$.
\end{definition}

\noindent Now, (\ref{BL-inequality}) can be generalized as follows.

\begin{theorem}\label{BL-generalized}
Let $0<r\leq\frac{\pi}{2}$ and ${\bf D}\in \BB^2_{s, {\rm wide}}(r)$. Then ${\rm area}_s({\bf D})\geq{\rm area}_s({\bf \Delta}_2(r))$.
\end{theorem}

\noindent As ${\bf \Delta}_2(r)\in\KK_s^2(r)\subset\BB^2_{s, {\rm wide}}(r)$ holds for all $0<r\leq\frac{\pi}{2}$ therefore (\ref{BL-inequality}) follows from Theorem~\ref{BL-generalized} in a straightforward way. We note that our method of proving Theorem~\ref{BL-generalized} is completely different from the ideas and techniques used in \cite{Le05}.

The rest of the paper is organized as follows. Section 3 gives a proof of Theorem~\ref{BL-generalized} via successive area decreasing cuts and symmetrization. That proof is based on some extremal properties of wide $r$-disk domains, which are discussed in Section 2. Furthermore, Section 2 investigates the higher dimensional analogue of wide $r$-disk domains called wide $r$-ball bodies for $0<r\leq\frac{\pi}{2}$. In particular, we determine their minimum spherical width (resp., inradius) in $\S^d$, $d\geq 2$. Also, it is shown that any minimum volume wide $r$-ball body is of constant width $r$ in $\S^d$, $d\geq 2$.

\section{On minimizing the inradius, width, and volume of wide $r$-ball bodies in $\S^d$ for $d\geq 2$ and $0<r\leq\frac{\pi}{2}$}

It is natural to extend Definition~\ref{wide-disk-domain-polygon} to $\S^d$ thereby introducing the family of wide $r$-ball bodies (resp., wide $r$-ball polyhedra) in $\S^d$ as follows.

\begin{definition}\label{wide-ball-body-polyhedron}
Let $0<r\leq\frac{\pi}{2}$ be given and let $\emptyset\neq X$ be a closed subset of $\S^d$, $d\geq 2$ with ${\rm diam}_s(X)\leq r$. Then 
${\bf B}_s^d[X, r]$
is called a wide $r$-ball body generated by $X$ in $\S^d$. The family of wide $r$-ball bodies of $\S^d$ is labelled by $\BB^d_{s,{\rm wide}}(r)$. If $X\subset\S^d$ with 
$0<{\rm card}(X)<+\infty$ and ${\rm diam}_s(X)\leq r$, then ${\bf B}_s^d[X, r]$ is called a wide $r$-ball polyhedron generated by $X$ in $\S^d$. The family of 
wide $r$-ball polyhedra of $\S^d$ is labelled by $\PP^d_{s, {\rm wide}}(r)$.  
\end{definition}

We leave the straightforward proof of the following claim (using Definition~\ref{wide-ball-body-polyhedron}) to the reader.

\begin{proposition}\label{approximation-general}
Every wide $r$-ball body ${\bf B}_s^d[X, r]\in\BB^d_{s, {\rm wide}}(r)$, $d\geq 2$, $0<r\leq\frac{\pi}{2}$ can be approximated (in the Hausdorff sense) arbitrarily close by a suitable wide $r$-ball polyhedron and therefore there exists a sequence
${\bf P}_n\in \PP^d_{s, {\rm wide}}(r)$, $n=1,2,\dots$ such that $\lim_{n\to +\infty}{\rm vol}_s({\bf P}_n)={\rm vol}_s({\bf B}_s^d[X, r])$, where ${\rm vol}_s(\cdot)$ stands for the $d$-dimensional  spherical volume of the corresponding set in $\S^d$.
\end{proposition}

Although the question of finding an extension of Theorem~\ref{BL-generalized} to $\S^d$ for $d\geq 3$ seems to be a natural one, it is a considerably more difficult problem than it appears at first sight. Based on Proposition~\ref{approximation-general} we can phrase it as follows.

\begin{prob}\label{BL-genaralized-in-d-dimensions}
Find
$$c_{BL}(r, d):=\inf\{{\rm vol}_s({\bf B}_s^d[X, r])\ |\ {\bf B}_s^d[X, r]\in\BB^d_{s, {\rm wide}}(r)\}$$ $$=\inf\{{\rm vol}_s({\bf P}_s^d[X, r])\ |\ {\bf P}_s^d[X, r]\in\PP^d_{s, {\rm wide}}(r)\}$$
for given $0<r\leq\frac{\pi}{2}$ and $d\ge 3$.
\end{prob}

%In connection with Problem~\ref{BL-genaralized-in-d-dimensions} it would be desirable to prove non-trivial lower bounds for $c_{BL}(r, d)$. In particular, it would be of interest to find spherical analogues of Schramm's lower bound (\cite{Sc88}) for the volume of convex bodies of constant width in $\E^d$. 
Proposition~\ref{Min-Inradius} can be used to lower bound $c_{BL}(r, d)$ with the spherical volume of a properly chosen ball. 

\begin{definition}\label{inradius-circumradius}
The smallest ball (resp., the largest ball) containing (resp., contained in) the convex body ${\bf K}\in\KK_s^d$, $d\ge 2$ is called the circumscribed (resp., inscribed) ball of ${\bf K}$ whose radius $R_{\rm cr}({\bf K})$ (resp., $R_{\rm in}({\bf K})$) is called the circumradius (resp., inradius) of ${\bf K}$.
\end{definition}

\begin{proposition}\label{Min-Inradius}
Let ${\bf B}_s^d[X, r]\in\BB^d_{s, {\rm wide}}(r)$ with $d\geq 2$ and $0<r\leq\frac{\pi}{2}$. Then $$R_{\rm in}({\bf B}_s^d[X, r]) \geq R_{\rm in}({\bf \Delta}_d(r)),$$
where ${\bf \Delta}_d(r)\in\BB^d_{s, {\rm wide}}(r)$ denotes the intersection of $d+1$ closed balls of radii $r$ centered at the vertices of a regular spherical $d$-simplex of edge length $r$ in $\S^d$.
\end{proposition} 

\begin{proof} 
As ${\bf B}_s^d[X, r]\in\BB^d_{s, {\rm wide}}(r)$ therefore ${\rm diam}_s(X)\leq r$. This and the spherical Jung theorem \cite{Dek95} imply that there exists $\xx_0\in \S^d$ such that $X\subset{\bf B}_s^d[\xx_0, R_{\rm cr}({\bf \Delta}_d(r))]$. It follows that $${\bf B}_s^d[\xx_0,R_{\rm in}({\bf \Delta}_d(r))]={\bf B}_s^d[\xx_0, r-R_{\rm cr}({\bf \Delta}_d(r))]\subset{\bf B}_s^d[X, r]$$ and therefore $R_{\rm in}({\bf B}_s^d[X, r]) \geq R_{\rm in}({\bf \Delta}_d(r))$, finishing the proof of Proposition~\ref{Min-Inradius}.
\end{proof}

For more details on the concepts introduced in Definitions~\ref{Lune},~\ref{Width-Lune}, and~\ref{Min-Spherical-Width}, we refer the interested reader to the recent paper of Lassak \cite{La15}. As usual, we say that the $(d-1)$-dimensional great sphere $C_s^{d-1}(\xx, \frac{\pi}{2})$ is a supporting $(d-1)$-dimensional great sphere of ${\bf K}\in\KK_s^d$ if $C_s^{d-1}(\xx, \frac{\pi}{2})\cap {\bf K}\neq\emptyset$ and ${\bf K}\subset {\bf B}_s^d[\xx,\frac{\pi}{2}]$, in which case ${\bf B}_s^d[\xx,\frac{\pi}{2}]$ is called a closed supporting hemisphere of ${\bf K}$. One can show that through each boundary point of ${\bf K}$ there exists at least one supporting $(d-1)$-dimensional great sphere of ${\bf K}$ moreover, ${\bf K}$ is the intersection of its closed supporting hemispheres. Two hemispheres of $\S^d$ are called opposite if the their centers are antipodes. 
\begin{definition}\label{Lune}
The intersection of two distinct closed hemispheres of $\S^d$ which are not opposite is called a lune of $\S^d$. Let $\LL_s^d$ denote the family of lunes in $\S^d$.
\end{definition}
\noindent Every lune of $\S^d$ is bounded by two $(d-1)$-dimensional hemispheres (lying on two distinct $(d-1)$-dimensional great spheres of $\S^d$) sharing pairs of antipodes in common, which are called the vertices of the lune. 
\begin{definition}\label{Width-Lune}
The angular measure of the angle formed by the two $(d-1)$-dimensional hemispheres bounding the lune ${\bf L}\in\LL_s^d$ (which is equal to spherical distance of the centers of the two $(d-1)$-dimensional hemispheres bounding ${\bf L}$) is called the spherical width of the given lune labelled by ${\rm width}_s({\bf L})$. 
\end{definition}

\begin{definition}\label{Min-Spherical-Width}
For every closed supporting hemisphere ${\bf H}$ of the convex body ${\bf K}\in \KK_s^d$ there exists a closed supporting hemisphere ${\bf H}'$ of ${\bf K}$ such that the lune ${\bf H}\cap{\bf H}'$ has minimal width for given ${\bf H}$ and ${\bf K}$. We call ${\rm width}_s({\bf H}\cap{\bf H}')$ the width of ${\bf K}$ determined by ${\bf H}$ and label it by ${\rm width}_{\bf H}({\bf K})$. Finally, the minimal spherical width (also called thickness) ${\rm width}_s({\bf K})$ of ${\bf K}$ is the smallest spherical width of the lunes that contain ${\bf K}$, i.e., ${\rm width}_s({\bf K})=\min\{{\rm width}_{\bf H}({\bf K})\ |\ {\bf H}\ {\it is\  a \ closed\  supporting\  hemisphere\  of}\ {\bf K}\}$.  
\end{definition}

\noindent Next, we recall the following claim from \cite{La15} (Claim 2), which is often applicable.

\begin{sublemma}\label{Min-Width-Lune}
Let ${\bf K}\in\KK_s^d$. If ${\bf L}\in\LL_s^d$ contains ${\bf K}$ and ${\rm width}_s({\bf K})={\rm width}_s({\bf L})$, then both centers of the $(d-1)$-dimensional hemispheres bounding ${\bf L}$ belong to ${\bf K}$.
\end{sublemma}

\begin{remark}
We note that Definition~\ref{Min-Spherical-Width} supports to say that the convex body ${\bf K}\in\KK_s^d$ is of constant width $0<w\leq \frac{\pi}{2}$ in $\S^d$ if the width of ${\bf K}$ with respect to any supporting hemisphere is equal to $w$. Theorem 2 of \cite{La20}  proves that this definition of constant width is equivalent to the one under Definition~\ref{core-concept}, i.e., ${\bf K}\in\KK_s^d(w)$ for $d\geq 2$ and $0<w\leq\frac{\pi}{2}$ if and only if ${\rm width}_s({\bf K})={\rm diam}_s({\bf K}) =w$ holds for ${\bf K}\in\KK_s^d$. 
\end{remark}

Now, we are ready to prove the following close relative of Proposition~\ref{Min-Inradius}.

\begin{proposition}\label{Min-Width}
Let ${\bf B}_s^d[X, r]\in\BB^d_{s, {\rm wide}}(r)$ with $d\geq 2$ and $0<r\leq\frac{\pi}{2}$. Then $${\rm width}_s({\bf B}_s^d[X, r]) \geq {\rm width}_s({\bf \Delta}_d(r))=r.$$
\end{proposition}

\begin{proof}
It will be convenient to use the following notion (resp., notation) from \cite{Be18}.
\begin{definition}\label{r-dual-body}
For a set $X\subseteq \S^d$, $d\geq 2$ and $0<r\leq\frac{\pi}{2}$ let the $r$-dual set $X^r$ of $X$ be defined by $X^r:={\bf B}_s^d[X, r]$. If the spherical interior ${\rm int}_s(X^r)\neq\emptyset$, then we call $X^r$ the $r$-dual body of $X$.
\end{definition}
\noindent $r$-dual sets satisfy some basic identities such as $((X^r)^r)^r=X^r$ and $(X\cup Y)^r=X^r\cap Y^r$, which hold for any $X\subseteq \S^d$ and $Y\subseteq \S^d$. Clearly, also monotonicity holds namely, $X\subseteq Y\subseteq \S^d$ implies $Y^r\subseteq X^r$. Thus, there is a good deal of similarity between $r$-dual sets and spherical polar sets in $\S^d$. For more details see \cite{Be18}. The following statement is a spherical analogue of Lemma 3.1 in \cite{blnp}.

\begin{sublemma}\label{support-property}
Let ${\bf H}$ be a closed supporting hemisphere of the $r$-dual body $X^r$ of $X\subset \S^d$ bounded by the $(d-1)$-dimensional great sphere $H$ in $\S^d$ such that ${\bf H}$ and $H$ support $X^r$ at the boundary point $\xx\in H\cap {\rm bd}(X^r)$, where $d\geq 2$, and $0<r\leq\frac{\pi}{2}$. Then the $d$-dimensional closed ball of radius $r$ of $\S^d$ that is tangent to $H$ at $\xx$ and lies in ${\bf H}$ contains the $r$-dual body $X^r$.
\end{sublemma}
\begin{proof} (The following proof is the spherical analogue of the Euclidean proof of Lemma 3.1 of \cite{blnp}.) Let ${\bf B}_s^d[\cc, r]$ be the $d$-dimensional closed ball of radius $r$ of $\S^d$ that is tangent to $H$ at $\xx$ and lies in ${\bf H}$. Assume that $X^r$ is not contained in ${\bf B}_s^d[\cc, r]$, i.e., let $\yy\in X^r\setminus {\bf B}_s^d[\cc, r]$. Then by taking the intersection of the configuration with the $2$-dimensional spherical plane spanned by $\xx, \yy$, and $\cc$  we see that there is a shorter circular arc of radius $r$ connecting $\xx$ and $\yy$ that is not contained in ${\bf B}_s^d[\cc, r]$ and therefore it is not supported by neither $H$ nor ${\bf H}$. On the other hand, as $\xx,\yy\in X^r$ therefore any such arc must be contained in $X^r$ and must be supported by $H$ as well as ${\bf H}$, a contradiction.
\end{proof}

Now, let $X^r={\bf B}_s^d[X, r]\in\BB^d_{s, {\rm wide}}(r)$ with $d\geq 2$ and $0<r\leq\frac{\pi}{2}$. Sublemma~\ref{Min-Width-Lune} implies that  there exists ${\bf L}\in\LL_s^d$ such that $X^r\subseteq {\bf L}:={\bf B}_s^d[\xx, \frac{\pi}{2}]\cap {\bf B}_s^d[\yy, \frac{\pi}{2}]$ and ${\bf x}'\in S_s^{d-1}(\xx, \frac{\pi}{2})\cap X^r$ is the center of the $(d-1)$-dimensional hemisphere $S_s^{d-1}(\xx, \frac{\pi}{2})\cap{\bf B}_s^d[\yy, \frac{\pi}{2}]$  and ${\bf y}'\in S_s^{d-1}(\yy, \frac{\pi}{2})\cap X^r$ is the center of the $(d-1)$-dimensional hemisphere $S_s^{d-1}(\yy, \frac{\pi}{2})\cap{\bf B}_s^d[\xx, \frac{\pi}{2}]$ satisfying ${\rm width}_s(X^r)={\rm width}_s({\bf L})={\rm dist}_s(\xx',\yy')$. It follows from Sublemma~\ref{support-property} in a straightforward way that there exists ${\bf B}_s^d[\xx'', r]$ (resp., ${\bf B}_s^d[\yy'', r]$) such that
$X^r\subseteq{\bf B}_s^d[\xx'', r]\subseteq {\bf B}_s^d[\xx, \frac{\pi}{2}]$ (resp., $X^r\subseteq{\bf B}_s^d[\yy'', r]\subseteq {\bf B}_s^d[\yy, \frac{\pi}{2}]$) and ${\bf B}_s^d[\xx'', r]$ (resp., ${\bf B}_s^d[\yy'', r]$) is tangent to $S_s^{d-1}(\xx, \frac{\pi}{2})$ (resp., $S_s^{d-1}(\yy, \frac{\pi}{2})$) at $\xx'$ (resp., $\yy'$) with $\xx''\in ({X^r})^{r}$ (resp., $\yy''\in ({X^r})^{r}$). By construction
$$2r-{\rm width}_s(X^r)=2r-{\rm dist}_s(\xx',\yy')={\rm dist}_s(\xx'',\yy'')\leq {\rm diam}_s\left((X^r)^{r}\right)$$
and therefore
\begin{equation}\label{Elso}
2r\leq {\rm width}_s(X^r)+{\rm diam}_s(({X^r})^{r}).
\end{equation}
\begin{sublemma}\label{diameter-bound}
Let $0<r\leq\frac{\pi}{2}$ be given and let $\emptyset\neq X$ be a closed subset of $\S^d$, $d\geq 2$ with ${\rm diam}_s(X)\leq r$. Then 
\begin{equation}\label{Masodik}
{\rm diam}_s\left((X^r)^{r}\right)\leq r.
\end{equation}
\end{sublemma}
\begin{proof}
Recall (\cite{De95} or \cite{La20}) that a closed set $Y\subset \S^d$ is called a complete set if ${\rm diam}_s(Y\cup\{\yy\})>{\rm diam}_s (Y)$ holds for all $\yy\in \S^d\setminus Y$. It is easy to prove the following claim (see Lemma 1 of \cite{La20}): if $Y$ is a complete set with ${\rm diam}(Y)\leq \frac{\pi}{2}$, then $Y= Y^{{\rm diam}_s(Y)}\in \KK_s^d({\rm diam}_s (Y))\subset \KK_s^d$. Furthermore, it is well know (see Theorem 1 of \cite{De95} or Theorem 1 of \cite{La20}) that each set of diameter $\delta \in (0, \pi)$ in $\S^d$ is a subset of a complete set of diameter $\delta$ in $\S^d$. Thus, there exists a complete set $Y\subset \S^d$ such that $X\subseteq Y$ with ${\rm diam}_s(X)\leq {\rm diam}_s(Y)=r$. By the monotonicity of the $r$-dual operation it follows that $(X^r)^r\subseteq (Y^r)^r=Y$ and so, ${\rm diam}_s\left((X^r)^{r}\right)\leq{\rm diam}_s(Y)=r$.
\end{proof}
Thus,  (\ref{Elso}) and (\ref{Masodik}) yield $r\leq {\rm width}_s(X^r)$, finishing the proof of Proposition~\ref{Min-Width}.
\end{proof}

In fact, $c_{BL}(r, d)$ is equal to the minimum of the volumes of convex bodies of constant width $r$ in $\S^d$ as stated in Proposition~\ref{reduction-corollary}. This can be proved as follows. Let $0<r\leq\frac{\pi}{2}$ be given and let $\emptyset\neq X$ be a closed subset of $\S^d$, $d\geq 2$ with ${\rm diam}_s(X)\leq r$. Then the proof of Sublemma~\ref{diameter-bound} shows the existence of a complete set $Y\subset \S^d$ such that $X\subseteq Y$ with ${\rm diam}_s(X)\leq {\rm diam}_s(Y)=r$. As $Y^r=Y$ therefore $Y$ is a convex body of constant width $r$, i.e., $Y\in\KK_s^d(r)$. Moreover, the monotonicity of the $r$-dual operation implies that $Y=Y^r\subseteq X^r$, where $X^r={\bf B}_s^d[X, r]\in\BB^d_{s, {\rm wide}}(r)$. Finally, Blaschke's selection theorem applied to $\KK_s^d(r)$ (\cite{Sch}) yields

%It is a straightforward corollary of Proposition~\ref{reduction}, which is the spherical analogue of the Peterson-Sallee construction of convex bodies of constant width in Euclidean spaces (\cite{Sa70}).

%{\color{red}
%\begin{proposition}\label{reduction}
%Let $0<r\leq\frac{\pi}{2}$ be given and let $\emptyset\neq X$ be a closed subset of $\S^d$, $d\geq 2$ with ${\rm diam}_s(X)\leq r$. Furthermore, let $Y:=\{\yy_n\ |\ n=1,2,\dots\}$ be a countable dense set in $\S^d$. If $\mathbf{F}_0:={\bf B}_s^d[X, r]$ and $\mathbf{F}_n:=\mathbf{F}_{n-1}\cap {\bf B}_s^d[\yy_n, r]$ for $\yy_n\in\mathbf{F}_{n-1}$ and $\mathbf{F}_n:=\mathbf{F}_{n-1}$ for $\yy_n\notin \mathbf{F}_{n-1}$, $n=1,2,\dots$, then $$X\subseteq \mathbf{W}_Y(X):=\cap\{\mathbf{F}_n\ |\ n=0,1,2,\dots\}\in \KK_{s}^d(r).$$
%\end{proposition}
%}

\begin{proposition}\label{reduction-corollary} 
Every wide $r$-ball-body ${\bf B}_s^d[X, r]\in\BB^d_{s, {\rm wide}}(r)$ contains a convex body of constant width $r$, i.e., there exists $Y={\bf B}_s^d[Y, r]\in\KK_{s}^d(r)$ such that $Y\subseteq {\bf B}_s^d[X, r]$, where $0<r\leq\frac{\pi}{2}$ and $d\ge 2$. Thus,
$$c_{BL}(r, d)=\min\{{\rm vol}_s({\bf K})\ |\ {\bf K}\in\KK_{s}^d(r)\}$$ 
holds for all $0<r\leq\frac{\pi}{2}$ and $d\ge 2$.
\end{proposition}

In connection with Proposition~\ref{reduction-corollary} it is natural to look for the spherical analogue of Schramm's lower bound (\cite{Sc88}) for the volume of convex bodies of constant width in $\E^d$. This has been done by Schramm (\cite{Schr}) for
$c_{BL}\left(\frac{\pi}{2}, d\right)=\min\left\{{\rm vol}({\bf K})\ |\ {\bf K}\in\KK_{s}^d\left(\frac{\pi}{2}\right)\right\}$ as follows.
%Here we recall that if $Y\subset\S^d$, then 
%$$Y^+:=\{\xx\in\S^d\ |\ \langle \xx ,\yy\rangle\geq0\ {\rm for}\ {\rm all}\ \yy\in Y\}={\bf B}_s^d\left[Y, \frac{\pi}{2}\right]$$ 
%is called the (spherical) polar of $Y$ in $\S^d$. Thus, we have
%\begin{remark}
%Clearly, ${\bf K}\in\KK_s^d(\frac{\pi}{2})$ if only if ${\bf K}$ is a {\it self-polar convex body} of $\S^d$, i.e., ${\bf K}^+={\bf K}\in \KK_s^d$.
%\end{remark}
\begin{remark}\label{special lower bound of Schramm}
Proposition 9 of \cite{Schr} implies that
$$c_{BL}\left(\frac{\pi}{2}, d\right)
\geq\sqrt{ \frac{8^d}{2\pi (d+1)(d+4)^d} }{\rm vol}_s\left({\Delta}_d\left(\frac{\pi}{2}\right)\right),$$ where ${\rm vol}_s(\S^d)=(d+1)\omega_{d+1}=\frac{(d+1)\pi^{\frac{d+1}{2}}}{\Gamma(\frac{d+3}{2})}$, ${\rm vol}_s\left({\Delta}_d\left(\frac{\pi}{2}\right)\right)=\frac{(d+1)\omega_{d+1}}{2^{d+1}}$ and $d\geq 3$.
\end{remark}

It seems reasonable to hope for the following strengthening of the estimate of Remark~\ref{special lower bound of Schramm} (resp., of Conjecture 1.6 from \cite{Be12}). 
%which if true can be viewed as an extension of the Blaschke-Leichtweiss theorem for constant width $\frac{\pi}{2}$ to higher dimensional spherical spaces.

\begin{con}\label{higher-dimensional-BL}
$c_{BL}(\frac{\pi}{2}, d)={\rm vol}_s({\Delta}_d(\frac{\pi}{2}))$, i.e., if
 ${\bf K}\in\KK_s^d(\frac{\pi}{2})$, then ${\rm vol}_s({\bf K})\geq{\rm vol}_s({\Delta}_d(\frac{\pi}{2}))$ for all $d\ge 3$.
\end{con}

\section{Proof of Theorem~\ref{BL-generalized}}

Let $0<r\leq\frac{\pi}{2}$ and ${\bf D}\in \BB^2_{s, {\rm wide}}(r)$. Our goal is to show that ${\rm area}_s({\bf D})\geq{\rm area}_s({\bf \Delta}_2(r))$. Let ${\bf C}_{\rm in}$ be the inscribed disk of ${\bf D}$ with center $\cc$ having radius $R_{\rm in}$. We may assume that $2R_{\rm in}< r$. Namely, if $r\leq2R_{\rm in}$, then ${\bf D}$ contains a disk of diameter $r$ and so, it follows via the spherical isodiametric inequality (\cite{BoSag}) that ${\rm area}_s({\bf D})\geq{\rm area}_s({\bf \Delta}_2(r))$. 

Next, one of the following two cases must occur: either the boundaries of ${\bf D}$ and ${\bf C}_{\rm in}$ have two points in common such that the shorter great circular arc connecting them is a diameter of ${\bf C}_{\rm in}$ or the boundaries of ${\bf D}$ and ${\bf C}_{\rm in}$ have three points in common such that $\cc$ is in the interior of the triangle that is the spherical convex hull of these three points. In the first case, Sublemma~\ref{support-property} implies that ${\bf D}$ contains a disk of diameter $r$ and so, as above we get that ${\rm area}_s({\bf D})\geq{\rm area}_s({\bf \Delta}_2(r))$. 

\begin{figure}[ht]
\begin{center}
\includegraphics[width=0.6\textwidth]{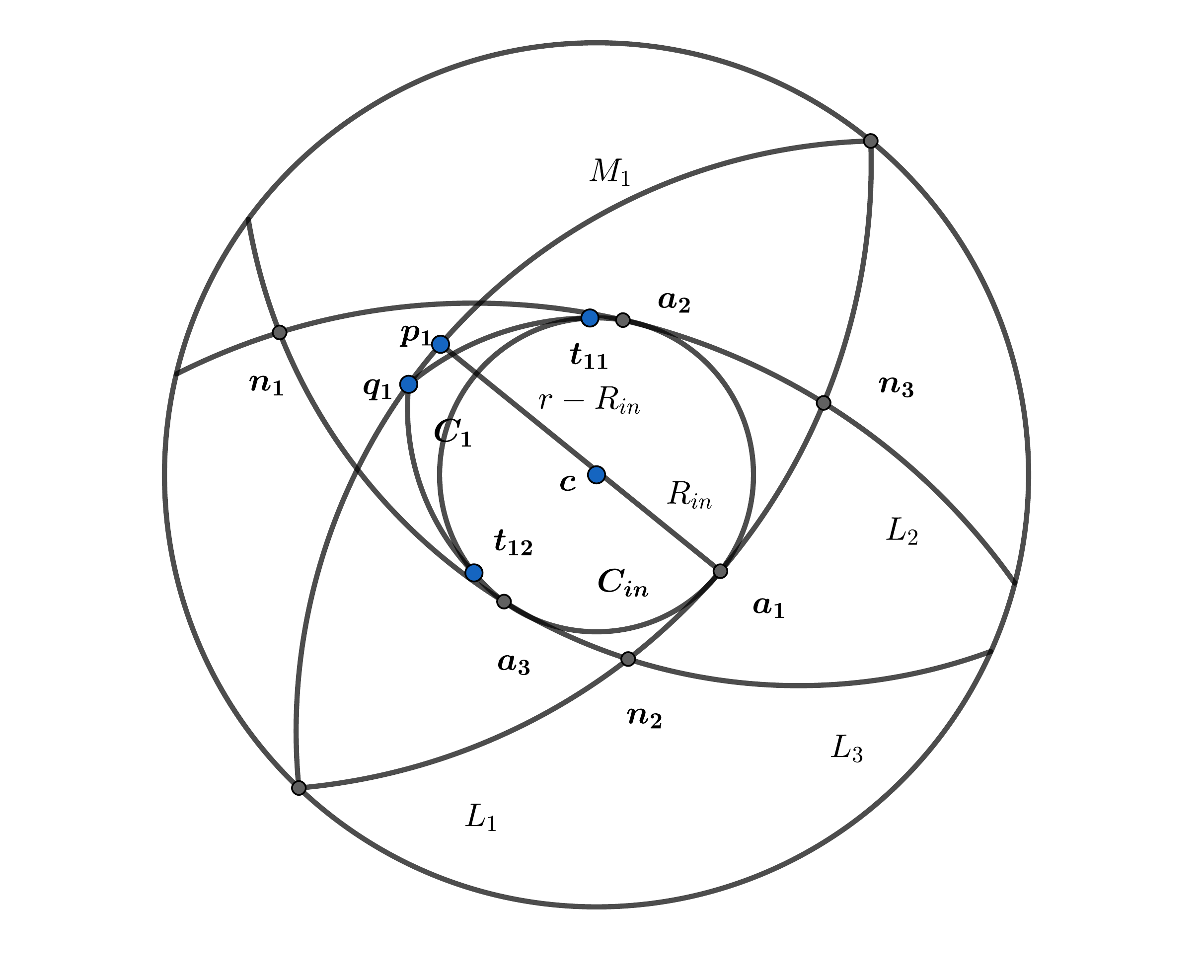}
\caption{Cap ${\bf C}_1$ of the cap-domain ${\bf C}:={\bf C}_1\cup{\bf C}_2\cup{\bf C}_3\cup {\bf C}_{\rm in}$ in the hemisphere of $\S^2$ with center $\cc$.}
\label{Figure1AAAA.pdf}
\end{center}
\end{figure}

In the second case, let the three selected points in common of the boundaries of ${\bf D}$ and ${\bf C}_{\rm in}$ be $\aa_1$, $\aa_2$, and $\aa_3$ and let the supporting great circles to ${\bf C}_{\rm in}$ at these points be $L_1$, $L_2$, and $L_3$ respectively (Figure~\ref{Figure1AAAA.pdf}). Note that $L_1$, $L_2$, and $L_3$ are also supporting great circles to ${\bf D}$ and thus, ${\bf D}$ lies in one of the spherical triangles determined by $L_1$, $L_2$, and $L_3$. Let us label this spherical triangle by $\triangle{\bf n}_1{\bf n}_2{\bf n}_3$ having the vertices ${\bf n}_1$, ${\bf n}_2$, and ${\bf n}_3$ such that that ${\bf n}_1$ is not on $L_1$ (i.e., ${\bf n}_1$ is ``opposite" to $L_1$ ), ${\bf n}_2$ is not on $L_2$, and ${\bf n}_3$ is not on $L_3$. Now, let $\pp_1$ be the point on the same side of $L_1$ as ${\bf D}$ such that $\pp_1\aa_1$ is of length $r$ and is perpendicular to $L_1$ at $\aa_1$. Let $M_1$ be the great circle perpendicular to $\pp_1\aa_1$ at $\pp_1$. Note that the angle between $L_1$ and $M_1$ is $r$. By Proposition~\ref{Min-Width}, $M_1$ must contain a point of ${\bf D}$. Let this point be $\qq_1$. Since $0<r\leq\frac{\pi}{2}$ we have that
\begin{equation}\label{negyedik} 
{\rm dist}_s(\qq_1,\cc)\geq{\rm dist}_s(\pp_1,\cc)=r-R_{\rm in}>R_{\rm in}.
\end{equation} 

Let $\tt_{11}$ and $\tt_{12}$ be the two points in common of the boundary of ${\bf C}_{\rm in}$ with the two circles of radius $r$ passing through $\qq_1$ that are tangent to ${\bf C}_{\rm in}$ and whose disks of radius $r$ contain ${\bf C}_{\rm in}$. Here the shorter circular arc of radius $r$ connecting $\qq_1$ and $\tt_{11}$ (resp., $\tt_{12}$) and sitting on the corresponding circle of radius $r$ just introduced, is labeled by $(\qq_1\tt_{11})_r$ (resp., $(\qq_1\tt_{12})_r$). The circular arcs $(\qq_1\tt_{11})_r$ and $(\qq_1\tt_{12})_r$ have equal lengths moreover, the cap ${\bf C}_1$ bounded by $(\qq_1\tt_{11})_r$ and $(\qq_1\tt_{12})_r$ and the shorter circular arc of radius $R_{\rm in}$ connecting $\tt_{11}$ and $\tt_{12}$ on the boundary of ${\bf C}_{\rm in}$ lies in ${\bf D}$ and therefore it lies also in the spherical triangle $\triangle{\bf n}_1\aa_2\aa_3\subset \triangle{\bf n}_1{\bf n}_2{\bf n}_3$. (Here we have used the property of ${\bf D}$ that if we choose two points in ${\bf D}$, then any shorter circular arc of radius $r'$ with $r\leq r'\leq\frac{\pi}{2}$ connecting the two points lies in ${\bf D}$.) We can perform the same procedure for the points $\aa_2$ and $\aa_3$, producing the caps ${\bf C}_2$ and ${\bf C}_3$. By construction the caps ${\bf C}_1$, ${\bf C}_2$, and ${\bf C}_3$ are non-overlapping and the cap-domain ${\bf C}:={\bf C}_1\cup{\bf C}_2\cup{\bf C}_3\cup {\bf C}_{\rm in}$ is a subset of ${\bf D}$ and therefore
\begin{equation}\label{otodik}
{\rm area}_s({\bf D})\geq{\rm area}_s({\bf C}).
\end{equation}

\begin{figure}[h]
\begin{center}
\includegraphics[width=0.6\textwidth]{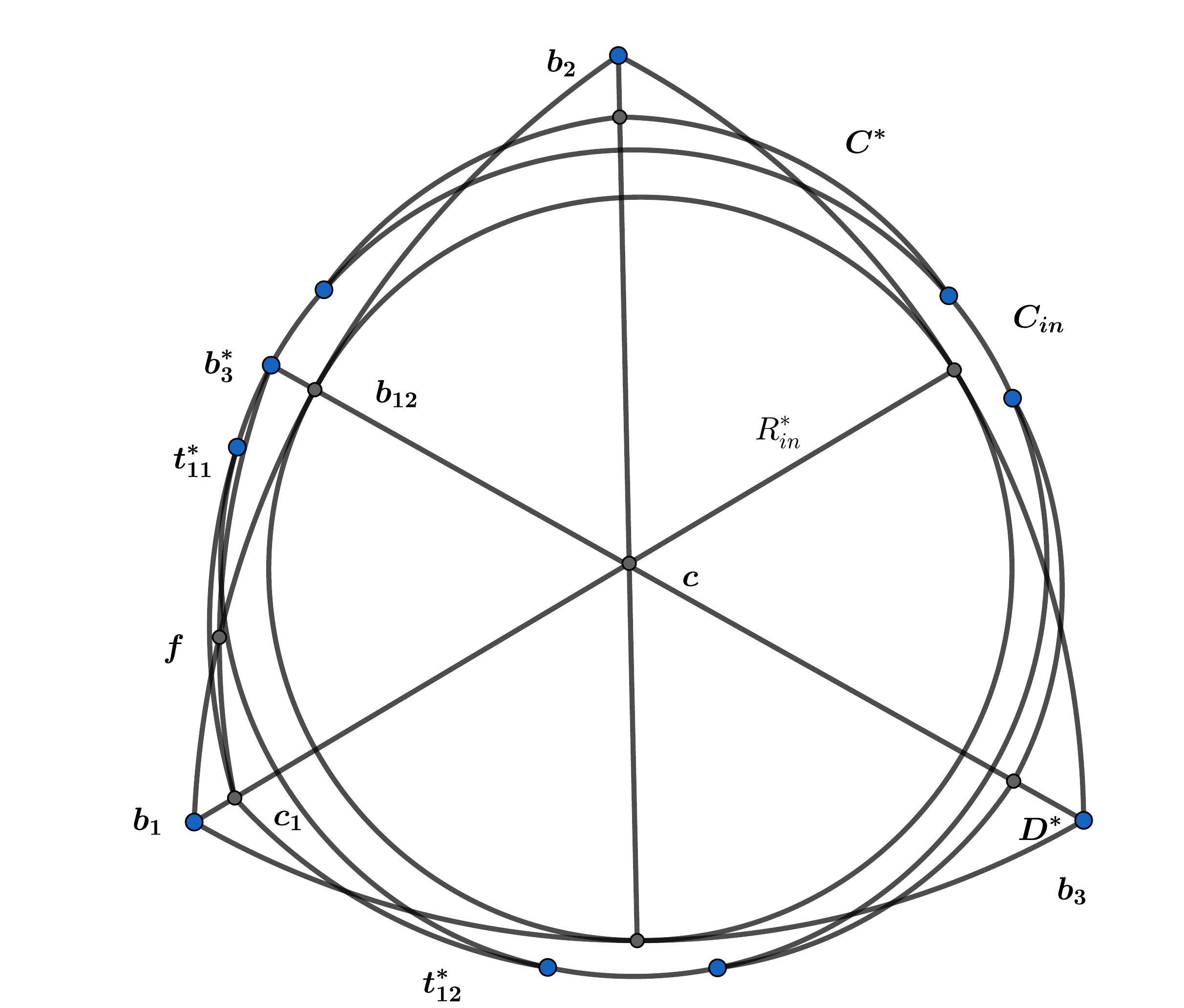}
\caption{The cap-domain ${\bf C}^*:={\bf C}_1^*\cup{\bf C}_2^*\cup{\bf C}_3^*\cup {\bf C}_{\rm in}$ compared to ${\bf D}^*$ via dissection and symmetry.}
\label{Figure2AAA.pdf}
\end{center}
\end{figure}

Let ${\bf D}^*:={\bf \Delta}_2(r)$ with vertices $\bb_1$, $\bb_2$, and $\bb_3$ such that its center is $\cc$ (Figure~\ref{Figure2AAA.pdf}). If $R_{\rm in}^*$ denotes the inradius of ${\bf D}^*$, then ${\rm dist}_s(\bb_1,\cc)={\rm dist}_s(\bb_2,\cc)={\rm dist}_s(\bb_3,\cc)=r-R_{\rm in}^*$. Clearly, Proposition~\ref{Min-Inradius} yields that $ r-R_{\rm in}^*\geq r-R_{\rm in}$. Thus, let $\cc_i$ be the point on the great circular arc $\bb_i\cc$ such that ${\rm dist}_s(\cc_i,\cc)=r-R_{\rm in}$, $1\leq i\leq 3$. Let $\tt_{11}^*$ and $\tt_{12}^*$ be the two points in common of the boundary of ${\bf C}_{\rm in}$ with the two circles of radius $r$ passing through $\cc_1$ that are tangent to ${\bf C}_{\rm in}$ and whose disks of radius $r$ contain ${\bf C}_{\rm in}$. Here the shorter circular arc of radius $r$ connecting $\cc_1$ and $\tt_{11}^*$ (resp., $\tt_{12}^*$) and sitting on the corresponding circle of radius $r$ just introduced, is labeled by $(\cc_1\tt_{11}^*)_r$ (resp., $(\cc_1\tt_{12}^*)_r$). The circular arcs $(\cc_1\tt_{11}^*)_r$ and $(\cc_1\tt_{12}^*)_r$ have equal lengths. Moreover, let the cap ${\bf C}_1^*$ be the domain bounded by $(\cc_1\tt_{11}^*)_r$ and $(\cc_1\tt_{12}^*)_r$ and the shorter circular arc of radius $R_{\rm in}$ connecting $\tt_{11}^*$ and $\tt_{12}^*$ on the boundary of ${\bf C}_{\rm in}$. From (\ref{negyedik}) it follows that ${\rm area}_s({\bf C}_1)\geq{\rm area}_s({\bf C}_1^*)$. Similarly, we can define the caps ${\bf C}_2^*$ and ${\bf C}_3^*$ with vertices $\cc_2$ and $\cc_3$ for which ${\rm area}_s({\bf C}_2)\geq{\rm area}_s({\bf C}_2^*)$ and ${\rm area}_s({\bf C}_3)\geq{\rm area}_s({\bf C}_3^*)$. By construction the caps ${\bf C}_1^*$, ${\bf C}_2^*$, and ${\bf C}_3^*$ are non-overlapping and therefore the cap-domain ${\bf C}^*:={\bf C}_1^*\cup{\bf C}_2^*\cup{\bf C}_3^*\cup {\bf C}_{\rm in}$ satisfies the inequality 
\begin{equation}\label{hatodik}
{\rm area}_s({\bf C})\geq{\rm area}_s({\bf C}^*).
\end{equation}
Based on (\ref{otodik}) and (\ref{hatodik}) we finish the proof of Theorem~\ref{BL-generalized} by showing the inequality 
\begin{equation}\label{final-inequality}
{\rm area}_s({\bf C}^*)\geq {\rm area}_s({\bf D}^*).
\end{equation}

Let $\bb_{12}$ be the midpoint of $(\bb_1\bb_2)_r$, which is the shorter circular arc of radius $r$ connecting $\bb_1$ and $\bb_2$ on the boundary of ${\bf D}^*$ (Figure~\ref{Figure2AAA.pdf}). Moreover, let $\bb_3^*:=\widehat{\cc\bb_3}\cap (S_s^1(\cc, R_{\rm in})\setminus \cc\bb_3)$. From this it follows that ${\rm dist}_s(\bb_1,\cc_1)={\rm dist}_s(\bb_{12},\bb_3^*)=R_{\rm in}-R_{\rm in}^*$. Let ${\bf f}:=(\bb_1\bb_{12})_r\cap(\cc_1\bb_3^*)_r$, where $(\bb_1\bb_{12})_r$ (resp., $(\cc_1\bb_3^*)_r$) is the shorter circular arc of radius $r$ connecting $\bb_1$ and $\bb_{12}$ (resp., $\cc_1$ and $\bb_3^*$) such that $(\bb_1\bb_{12})_r$ lies on the boundary of ${\bf D}^*$ (resp., the disk ${\bf B}_s^2[\cc', r]$ containing $(\cc_1\bb_3^*)_r$ on its boundary contains $\bb_3$ (resp., $\bb_1$) in its interior (resp., exterior)). We note that by construction 
\begin{equation}\label{hetedik}
(\cc_1\bb_3^*)_r\subset {\bf C}^* .
\end{equation}

\begin{sublemma}\label{Cauchy-applied}
Let ${\bf u}:=\widehat{\bb_3\cc'}\cap(S_s^1(\bb_3, r)\setminus {\bf B}_s^2[\cc', r])$ and ${\bf v}:=\widehat{\bb_3\cc'}\cap(S_s^1(\cc', r)\setminus {\bf B}_s^2[\bb_3, r])$ (Figure~\ref{Figure3AA.pdf}). Furthermore, let $({\bf f}{\bf v})_r$ be the shorter circular arc of $S_s^1(\cc', r)$ connecting ${\bf f}$ and ${\bf v}$ and let $\xx\in({\bf f}{\bf v})_r$ be a point moving from ${\bf f}$ to ${\bf v}$.  Then the point of ${\bf B}_s^2[\bb_3, r])$ closest to $\xx$ is $\yy:=\bb_3\xx\cap S_s^1(\bb_3, r)$ and ${\rm dist}_s (\xx,\yy)$ is a strictly increasing function of the length of $({\bf f}\xx)_r$, where $({\bf f}\xx)_r$ denotes the shorter circular arc of $S_s^1(\cc', r)$ connecting ${\bf f}$ and $\xx$.
\end{sublemma}

\begin{figure}[h]
\begin{center}
\includegraphics[width=0.5\textwidth]{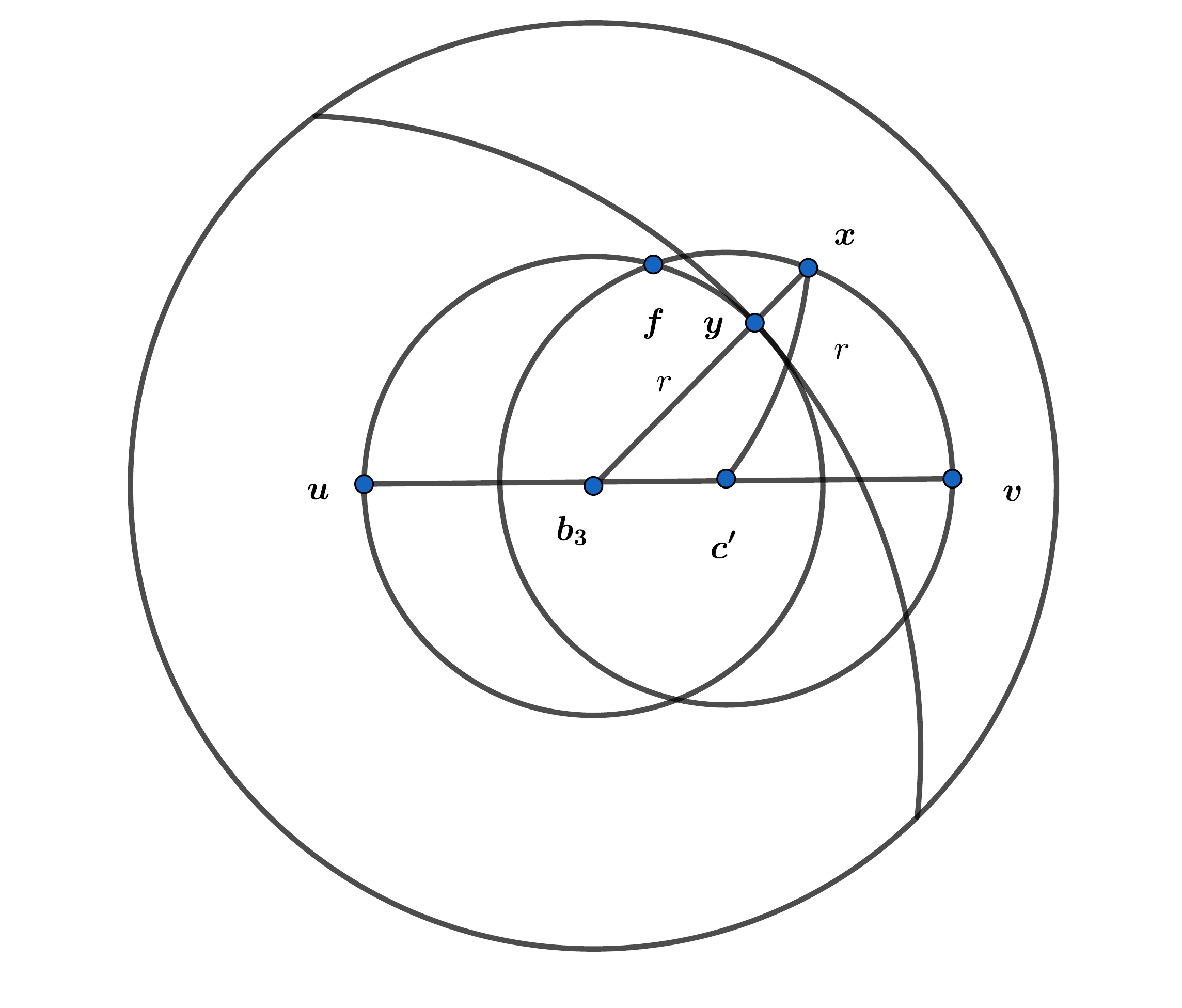}
\caption{Cauchy's Arm Lemma applied to $\triangle\bb_3\cc'\xx$ in $\S^2$}
\label{Figure3AA.pdf}
\end{center}
\end{figure}

\begin{proof}
Clearly, the point of ${\bf B}_s^2[\bb_3, r]$ closest to $\xx$ must have the property that the great circle passing through it and tangent to ${\bf B}_s^2[\bb_3, r]$ is orthogonal to the great circular arc connecting that point to $\xx$. It follows that the closest point is $\yy=\bb_3\xx\cap S_s^1(\bb_3, r)$. On the other hand, notice that as $\xx\in({\bf f}{\bf v})_r$ moves from ${\bf f}$ to ${\bf v}$ the angle $\angle\bb_3\cc'\xx$ at the vertex $\cc'$ of the spherical triangle $\triangle\bb_3\cc'\xx$ (bounded by the great circular arcs $\bb_3\cc', \cc'\xx$ and $\bb_3\xx$) strictly increases and so,  the spherical version of Cauchy's Arm Lemma (see \cite{Cauchy} or \cite{Cro}, p. 228) implies that ${\rm dist}_s (\bb_3,\xx)$
strictly increases and therefore also ${\rm dist}_s (\xx,\yy)={\rm dist}_s (\bb_3,\xx)-r$ strictly increases.
\end{proof}

Next, we note that the spherical distace of $\bb_3^*$ (resp., $\bb_1$) to ${\bf B}_s^2[\bb_3, r]$ (resp., ${\bf B}_s^2[\cc', r]$) is equal to ${\rm dist}_s(\bb_3^*, \bb_{12})=R_{\rm in}-R_{\rm in}^*$ (resp., is at most ${\rm dist}_s(\bb_1,\cc_1)=R_{\rm in}-R_{\rm in}^*$). Hence, Sublemma~\ref{Cauchy-applied} implies that the length of $(\bb_1{\bf f})_r$ (resp., $(\cc_1{\bf f})_r$) is at most as large as the length of $(\bb_3^*{\bf f})_r$ (resp., $(\bb_{12}{\bf f})_r$), where $(\bb_1{\bf f})_r$, $(\cc_1{\bf f})_r$, $(\bb_3^*{\bf f})_r$, and $(\bb_{12}{\bf f})_r$ are circular arcs of radius $r$ with endpoints indicated such that $(\bb_1{\bf f})_r\subset({\bf b}_1{\bf b}_{12})_r$, $(\cc_1{\bf f})_r\subset(\cc_1\bb_3^*)_r$, $(\bb_3^*{\bf f})_r\subset(\cc_1\bb_3^*)_r$, and $(\bb_{12}{\bf f})_r\subset({\bf b}_1{\bf b}_{12})_r$. It follows that the triangular shape region $\widehat{\triangle}\bb_1\cc_1{\bf f}$ bounded by $(\bb_1{\bf f})_r, \bb_1\cc_1$, and $(\cc_1{\bf f})_r$ has an isometric copy contained in the triangle shape region $\widehat{\triangle}\bb_3^*\bb_{12}{\bf f}$ bounded by $(\bb_3^*{\bf f})_r, \bb_3^*\bb_{12}$, and $(\bb_{12}{\bf f})_r$. This implies that 
\begin{equation}\label{nyolcadik}
{\rm area}_s(\widehat{\triangle}\bb_1\cc_1{\bf f})\leq{\rm area}_s(\widehat{\triangle}\bb_3^*\bb_{12}{\bf f}).
\end{equation} 
Thus, using (\ref{hetedik}), (\ref{nyolcadik}), and the symmetries of ${\bf D}^*$ and ${\bf C}^*$ we get that
\begin{equation}\label{kilencedik}
\frac{1}{6}{\rm area}_s({\bf D}^*\setminus{\bf C}^*)\leq {\rm area}_s(\widehat{\triangle}\bb_1\cc_1{\bf f})\leq{\rm area}_s(\widehat{\triangle}\bb_3^*\bb_{12}{\bf f})\leq \frac{1}{6}{\rm area}_s({\bf C}^*\setminus{\bf D}^*).
\end{equation}
Hence, (\ref{final-inequality}) follows, finishing the proof of Theorem~\ref{BL-generalized}.

\bigskip

\noindent K\'aroly Bezdek \\
\small{Department of Mathematics and Statistics, University of Calgary, Calgary, Canada}\\
\small{Department of Mathematics, University of Pannonia, Veszpr\'em, Hungary}\\
\small{\texttt{bezdek@math.ucalgary.ca}}

\end{document}